\documentclass{proc-l}

\usepackage{amsmath,amssymb,amsfonts}
\usepackage{mathtools}
\usepackage{mathrsfs}
\usepackage{enumerate}
\usepackage[hidelinks]{hyperref}

\allowdisplaybreaks[2]

\numberwithin{equation}{section}

\newtheorem{theorem}{Theorem}[section]
\newtheorem{lemma}[theorem]{Lemma}
\newtheorem{proposition}[theorem]{Proposition}
\newtheorem{corollary}[theorem]{Corollary}

\newtheorem*{theoremA}{Theorem A}

\theoremstyle{definition}

\theoremstyle{remark}

\newcommand{\D}{\mathbb D}
\newcommand{\T}{\mathbb T}
\newcommand{\C}{\mathbb C}
\newcommand{\BMOA}{\mathrm{BMOA}}

\newcommand{\dm}{\,dm}
\newcommand{\dA}{\,dA}
\newcommand{\norm}[1]{\left\lVert #1\right\rVert}
\newcommand{\abs}[1]{\left|#1\right|}
\newcommand{\BMOApNorm}[2]{\left\lVert #1\right\rVert_{\BMOA,#2}}
\DeclareMathOperator{\sgn}{sgn}

\title[Isometric Composition Operators on BMOA]
{Isometric Composition Operators on \(\BMOA\) for \(0<p<\infty\)}
 
 \author{ 
Zhaopeng Lin 
}



\begin{document}
 
\begin{abstract}
We characterize the analytic self-maps of the unit disk inducing
isometric composition operators on \(\BMOA\) with respect to the
M\"obius-invariant \(H^p\) norm for \(p\ge1\) and the corresponding
quasi-norm for \(0<p<1\). Extending the results of Pouliasis for
\(1\le p<2\) and Chen and Wulan for \(1\le p\le4\), we settle the
remaining range \(p>4\) and thus completely resolve Laitila's problem,
while also covering the range \(0<p<1\). Consequently, the isometric
property of \(C_\varphi\) is independent of \(p\in(0,\infty)\).
\end{abstract} 

\maketitle

\section{Introduction}\label{sec:introduction}

Let
\[
    \D=\{z\in\C:|z|<1\},
    \qquad
    \T=\partial\D,
\]
and let \(H(\D)\) denote the space of analytic functions on \(\D\).
If \(\varphi:\D\to\D\) is an analytic self-map, then the composition
operator induced by \(\varphi\) is defined by
\[
    C_\varphi f=f\circ\varphi,
    \qquad f\in H(\D).
\]

Composition operators constitute a class of operators whose
operator-theoretic properties are closely connected with the
function-theoretic behavior of their symbols. Isometric composition operators have been studied on several classical
spaces of analytic functions, including Hardy, Dirichlet, Bloch,
Besov-type, and weighted spaces; see, for example,
\cite{MartinVukoticDirichlet,ColonnaBloch,MartinVukoticBloch,
AllenHellerPons,ShabazzTjani} and the references therein.
Their characterization depends essentially on both the underlying
space and the norm under consideration.

We now turn to the space \(\BMOA\). Let \(dm\) denote normalized
Lebesgue measure on \(\T\), and let \(dA\) denote unnormalized planar area measure on \(\D\). For \(0<p<\infty\), the
Hardy space \(H^p\) consists of the analytic functions \(f\) on
\(\D\) for which
\[
    \norm{f}_{H^p}
    :=
    \sup_{0<r<1}
    \left(
        \int_{\T}|f(r\zeta)|^p\dm(\zeta)
    \right)^{1/p}
    <\infty.
\]

For \(a\in\D\), let
\[
    \sigma_a(z)
    =
    \frac{a-z}{1-\overline a z},
    \qquad z\in\D,
\]
be the involutive automorphism of \(\D\) interchanging \(0\) and
\(a\). For \(f\in H(\D)\), \(a\in\D\), and \(0<p<\infty\), set
\begin{equation}\label{eq:A-definition}
    A_p(f,a)
    :=
    \norm{f\circ\sigma_a-f(a)}_{H^p}^{p}
    =
    \int_{\T}
    |f(\sigma_a(\zeta))-f(a)|^p\dm(\zeta),
\end{equation}
and
\begin{equation}\label{eq:seminorm}
    B_p(f)
    :=
    \sup_{a\in\D}A_p(f,a)^{1/p}.
\end{equation}
An analytic function \(f\) belongs to \(\BMOA\) if and only if
\(B_p(f)<\infty\) for one, equivalently for every, \(p>0\). For
\(p\ge1\),
\begin{equation}\label{eq:BMOA-p-norm}
    \BMOApNorm{f}{p}
    :=
    |f(0)|+B_p(f)
\end{equation}
defines an equivalent norm on \(\BMOA\), whereas for \(0<p<1\) the
same formula defines an equivalent quasi-norm. Throughout the paper, isometry always refers to preservation of the norm \(\BMOApNorm{\cdot}{p}\) when \(p\ge1\), and of the corresponding quasi-norm when \(0<p<1\).

It is useful to recall that the isometric composition operators on the
Hardy space \(H^2\) were characterized by Ryff and Nordgren
\cite{Ryff,Nordgren}:
\[
C_\varphi \text{ is an isometry on } H^2
\quad\Longleftrightarrow\quad
\varphi \text{ is inner and } \varphi(0)=0.
\]
In contrast, the situation on \(\BMOA\) is less rigid. Shapiro
\cite{ShapiroBMOA}, see also Kobayashi \cite{Kobayashi1989}, showed
that every inner self-map \(\varphi\) with \(\varphi(0)=0\) induces an
isometric composition operator on \(\BMOA\), whereas Kobayashi
\cite{Kobayashi1990} constructed non-inner symbols inducing
isometries on \(\BMOA\). Thus, unlike the Hardy-space case, innerness
does not characterize isometric composition operators on \(\BMOA\).

A systematic characterization for the M\"obius-invariant \(H^2\) norm
was obtained by Laitila \cite{Laitila}. For \(a\in\D\), set $
\varphi_a
=
\sigma_{\varphi(a)}\circ\varphi\circ\sigma_a$. 
For \(\varphi(0)=0\), Laitila characterized isometric composition
operators in terms of sequences \((a_n)\subset\D\) satisfying
\[
\varphi(a_n)\to w,
\qquad
\|\varphi_{a_n}\|_{H^2}\to1
\]
for every \(w\in\D\), and obtained an equivalent characterization using
the test functions \(\sigma_w\). He further asked
\cite[Remark~4.4]{Laitila} for an analogous characterization with
respect to the M\"obius-invariant \(H^p\) norms for
\(1\le p<\infty\), \(p\ne2\).

Pouliasis \cite{Pouliasis} subsequently treated the range
\(1\le p<2\). For \(\varphi(0)=0\), he characterized isometric
composition operators by the condition that, for every \(w\in\D\),
there exists a sequence \((a_n)\subset\D\) such that
\[
\varphi(a_n)\to w,
\qquad
H_\varphi(z,a_n)\to0
\quad\text{for almost every }z\in\D,
\]
where \(H_\varphi\) denotes the Lindel\"of--Nevanlinna defect.
In particular, the class of isometric symbols in this range
coincides with that for \(p=2\).

To formulate Pouliasis' condition more explicitly, for \(z,b\in\D\)
let
\begin{equation}\label{eq:green-function}
G(z,b)
:=
\log\left|
\frac{1-\overline b z}{z-b}
\right|
\end{equation}
be the Green function of \(\D\). For an analytic self-map
\(\varphi:\D\to\D\), define the Nevanlinna counting function with
base point \(a\) by
\begin{equation}\label{eq:counting}
N_\varphi(z,a)
:=
\sum_{\varphi(u)=z}G(u,a)
=
\sum_{\varphi(u)=z}
\log\left|
\frac{1-\overline a u}{u-a}
\right|,
\end{equation}
where the preimages are counted with multiplicity. The corresponding
Lindel\"of--Nevanlinna defect is
\begin{equation}\label{eq:defect-definition}
H_\varphi(z,a)
:=
G(z,\varphi(a))-N_\varphi(z,a)
\ge0,
\end{equation}
where the inequality follows from the Lindel\"of principle.

More recently, Chen and Wulan \cite{ChenWulan} extended the
characterization to the range $
1\le p\le4$. 
In particular, they showed that for \(1\le p\le4\) the
disk automorphisms \(\sigma_w\) form an appropriate test family.
Their theorem may be stated as follows.

\begin{theoremA}\cite[Theorem~1.1]{ChenWulan}
Let \(1\le p\le4\), and let \(\varphi:\D\to\D\) be an analytic
self-map satisfying \(\varphi(0)=0\). Then the following conditions
are equivalent.
\begin{enumerate}[(i)]
    \item The composition operator \(C_\varphi\) is an isometry with
    respect to \(\BMOApNorm{\cdot}{p}\); that is,
    \[
        \BMOApNorm{f\circ\varphi}{p}
        =
        \BMOApNorm{f}{p}
        \qquad(f\in\BMOA).
    \]

    \item For every \(w\in\D\),
    \[
        B_p(\sigma_w\circ\varphi)=1.
    \]

    \item For every \(w\in\D\), there exists a sequence
    \(\{a_n\}_{n\ge 1}\subset\D\) such that
    \[
        \varphi(a_n)\longrightarrow w,
        \qquad
        \norm{\varphi_{a_n}}_{H^p}\longrightarrow1.
    \]

    \item For every \(w\in\D\), there exists a sequence
    \(\{a_n\}_{n\ge 1}\subset\D\) such that
    \[
        \varphi(a_n)\longrightarrow w,
        \qquad
        H_\varphi(z,a_n)\longrightarrow0
    \]
    for almost every \(z\in\D\).
\end{enumerate}
\end{theoremA}

The restriction \(p\le4\) is sharp for the disk-automorphism
testing condition in Theorem~A: Chen and Wulan showed that this
condition fails for \(p>4\). Thus Laitila's problem remained open
in the range \(p>4\), while the quasi-Banach range \(0<p<1\) was
not covered by the preceding results.

The purpose of the present paper is to complete the characterization
for the full range \(0<p<\infty\).  To explain the obstruction, for \(p>0\) define
\begin{equation}\label{eq:radial-function}
    \mathcal A_p(r)
    :=
    A_p(\sigma_0,r)
    =
    (1-r^2)^p
    \int_{\T}|1-r\zeta|^{-p}\dm(\zeta),
    \qquad 0\le r<1.
\end{equation}
Since \(\sigma_0(z)=-z\), multiplication by a unimodular constant
shows that $
    A_p(\sigma_0,r)=A_p(z,r)$.  

The obstruction for \(p>4\) is already visible from
\begin{equation}\label{eq:Mobius-function-intro}
A_p(\sigma_w,a)
=
\mathcal A_p\bigl(|\sigma_a(w)|\bigr).
\end{equation}
For \(0<p\le4\), \(\mathcal A_p\) has its unique maximum at \(0\), so
\(a=w\) is the unique maximizing center. For \(p>4\), however,
\(\mathcal A_p\) has a unique maximizing radius \(r_p\in(0,1)\);
hence the maximizing centers form the pseudohyperbolic circle $
|\sigma_a(w)|=r_p$. 
Thus \(\sigma_w\) no longer singles out \(w\). Our first result
constructs a replacement test function with a prescribed unique
maximizing center.

\begin{theorem}\label{thm:peak-family}
Let \(0<p<\infty\). For every \(w\in\D\), there exists a function
\(\mathcal P_{p,w}\), analytic on a neighborhood of
\(\overline\D\), such that
\[
    A_p(\mathcal P_{p,w},w)=1
    >
    A_p(\mathcal P_{p,w},a),
    \qquad a\in\D\setminus\{w\}.
\]
In particular,
\[
    B_p(\mathcal P_{p,w})=1.
\]
More precisely, the functions \(\mathcal P_{p,w}\) may be chosen as
follows.
\begin{enumerate}[(i)]
    \item If \(0<p\le4\), then one may take $
        \mathcal P_{p,w}=\sigma_w$. 

    \item If \(p>4\), then there exist numbers
    \[
        \tau_p>0,
        \qquad
        \rho_p\in(0,1),
    \]
    such that, for
    \[
        F_p(z)=z+\tau_pz^3+\tau_pz^4,
    \]
    the function $
        a\longmapsto A_p(F_p,a)$ 
    has a unique nondegenerate global maximum at \(a=\rho_p\).
    If
    \[
        \Lambda_p
        :=
        A_p(F_p,\rho_p)^{1/p}
        =
        B_p(F_p),
    \]
    then one may take
\begin{equation}\label{eq:high-p-peak}
\mathcal P_{p,w}(z)
=
\Lambda_p^{-1}
F_p\!\left(\sigma_{\rho_p}(\sigma_w(z))\right).
\end{equation}
\end{enumerate}
\end{theorem}

The resulting test functions give the following complete
characterization. For each \(p>0\) and \(w\in\D\), fix a function
\(\mathcal P_{p,w}\) as in Theorem~\ref{thm:peak-family}.

\begin{theorem}\label{thm:main}
Let \(0<p<\infty\), and let \(\varphi:\D\to\D\) be an analytic
self-map satisfying \(\varphi(0)=0\). Then the following conditions
are equivalent.
\begin{enumerate}[(i)]
    \item The composition operator \(C_\varphi\) is an isometry with
    respect to \(\BMOApNorm{\cdot}{p}\); that is,
    \[
        \BMOApNorm{f\circ\varphi}{p}
        =
        \BMOApNorm{f}{p}
        \qquad(f\in\BMOA).
    \]

\item For every \(w\in\D\),
\[
    B_p\bigl(\mathcal P_{p,w}\circ\varphi\bigr)=1.
\]

    \item For every \(w\in\D\), there exists a sequence
    \(\{a_n\}_{n\ge 1}\subset\D\) such that
    \[
        \varphi(a_n)\longrightarrow w,
        \qquad
        \norm{\varphi_{a_n}}_{H^p}\longrightarrow1.
    \]

    \item For every \(w\in\D\), there exists a sequence
    \(\{a_n\}_{n\ge 1}\subset\D\) such that
    \[
        \varphi(a_n)\longrightarrow w,
        \qquad
        H_\varphi(z,a_n)\longrightarrow0
    \]
    for almost every \(z\in\D\).
\end{enumerate}
\end{theorem}

The corollary follows immediately from Theorem~\ref{thm:main},
since condition \textup{(iv)} is independent of \(p\).

\begin{corollary}\label{cor:p-independence}
Let \(\varphi:\D\to\D\) be analytic. Then the isometric property of
\(C_\varphi\) with respect to \(\BMOApNorm{\cdot}{p}\) is independent
of \(p\in(0,\infty)\).
\end{corollary}

\section{Preliminaries}\label{sec:preliminaries}

We shall use the following elementary covariance property.

\begin{lemma}\label{lem:covariance}
Let \(0<p<\infty\), let \(f\in\BMOA\), and let \(\tau\) be an
automorphism of \(\D\). Then
\begin{equation}\label{eq:covariance}
    A_p(f\circ\tau,a)=A_p(f,\tau(a)),
    \qquad a\in\D.
\end{equation}
\end{lemma}

\begin{proof}
For each \(a\in\D\), the automorphism $
    \sigma_{\tau(a)}\circ\tau\circ\sigma_a$ 
fixes the origin. Hence there exists \(\lambda\in\T\) such that
\[
    \sigma_{\tau(a)}\circ\tau\circ\sigma_a(z)=\lambda z,
\]
or equivalently,
\[
    \tau\circ\sigma_a
    =
    \sigma_{\tau(a)}\circ R_\lambda,
    \qquad
    R_\lambda(z)=\lambda z.
\]
Therefore, by the rotational invariance of the \(H^p\)-quantity,
\[
\begin{aligned}
A_p(f\circ\tau,a)
&=
\norm{f\circ\tau\circ\sigma_a-f(\tau(a))}_{H^p}^{p}\\
&=
\norm{
    \bigl(f\circ\sigma_{\tau(a)}-f(\tau(a))\bigr)
    \circ R_\lambda
}_{H^p}^{p}\\
&=
A_p(f,\tau(a)).
\end{aligned}
\]
\end{proof}
 For the corresponding computation involving the disk automorphisms
\(\sigma_w\), see \cite[proof of Theorem~3.1]{Laitila},
\cite[(2.4)]{ChenWulan}, or \cite[(3.8)]{Pouliasis}.

 Recall from \eqref{eq:green-function}--\eqref{eq:defect-definition}
the Green function \(G\), the Nevanlinna counting function
\(N_\varphi\), and the Lindel\"of--Nevanlinna defect \(H_\varphi\). 
The same computation based on the Hardy--Stein identity, the
non-univalent change-of-variables formula, and the Lindel\"of
principle yields the following identity for every \(0<p<\infty\);
compare \cite[(1.2), pp.~2518--2520]{ChenWulan} and
\cite[Section~4]{Pouliasis}. For \(0<p<\infty\), \(f\in\BMOA\), and
\(a\in\D\),
\begin{equation}\label{eq:defect-formula}
 A_p(f\circ\varphi,a)
 =
 A_p(f,\varphi(a))-\Delta_{p,f,\varphi}(a),
\end{equation}
where
\begin{equation}\label{eq:defect-integral}
 \Delta_{p,f,\varphi}(a)
 =
 \frac{p^2}{2\pi}
 \int_{\D}
 \abs{f(z)-f(\varphi(a))}^{p-2}
 \abs{f'(z)}^2
 H_\varphi(z,a)\dA(z)
 \ge0.
\end{equation} 
For \(p<2\), the apparent singularities at the zeros of
\(f-f(\varphi(a))\) are locally integrable; indeed, near a zero of
order \(m\), the corresponding density is of order
\(|z-z_0|^{mp-2}\), up to the logarithmic factor coming from the
Green function.

Since \(\Delta_{p,f,\varphi}(a)\ge0\),
\eqref{eq:defect-formula} immediately gives
\[
A_p(f\circ\varphi,a)
\le
A_p(f,\varphi(a))
\le
B_p(f)^p.
\]
Taking the supremum over \(a\in\D\), we obtain
\begin{equation}\label{eq:subordination}
B_p(f\circ\varphi)\le B_p(f).
\end{equation}
See \cite[(2.3)]{ChenWulan} and \cite[(3.1)]{Pouliasis}.

\begin{proposition} \label{prop:seminorm-reduction}
Let \(0<p<\infty\). 
The operator \(C_\varphi\) is an isometry for
\(\BMOApNorm{\cdot}{p}\) if and only if
\[
 \varphi(0)=0
 \quad\text{and}\quad
 B_p(f\circ\varphi)=B_p(f)
 \quad\text{for every }f\in\BMOA.
\]
\end{proposition}

\begin{proof}
Conversely, equality in \eqref{eq:BMOA-p-norm} and
\eqref{eq:subordination} give 
\(\abs{f(0)}\le\abs{f(\varphi(0))}\) for every \(f\in\BMOA\).
Taking \(f=\sigma_{\varphi(0)}\) yields \(\varphi(0)=0\), after which
equality of the full quasi-norms reduces to equality of the oscillation quasi-norms.
\end{proof}

\begin{lemma}
\label{lem:exponent-transfer}
Let \((g_n)\) be analytic self-maps of \(\D\). If $
 \norm{g_n}_{H^{s_0}}\to1$ 
for some \(s_0>0\), then $
 \norm{g_n}_{H^s}\to1$ 
for every \(s>0\).
\end{lemma}

\begin{proof}
Put \(X_n=\abs{g_n^*}\le1\) almost everywhere. If \(0<s\le s_0\),
then \(X_n^s\ge X_n^{s_0}\). If \(s>s_0\), then
\[
 1-X_n^s\le\frac{s}{s_0}(1-X_n^{s_0}).
\]
The conclusion follows by integration.
\end{proof}

\section{The range \texorpdfstring{\(p>4\)}{p > 4}: radial function and splitting coefficients}\label{sec:high-p-analysis}

\subsection{Extremal behavior of
\texorpdfstring{\(\mathcal A_p\)}{the radial function}}

Fix \(p>4\) and recall the radial function
\(\mathcal A_p\) from \eqref{eq:radial-function}. Set
\begin{equation}\label{eq:beta}
\beta=\frac{p-2}{2}>1,
\qquad
x=r^2,
\qquad
Y_p(x)=\mathcal A_p(\sqrt{x}).
\end{equation}
The hypergeometric expansion and Euler transformation give
\begin{equation}\label{eq:Y-hypergeometric}
Y_p(x)
=
(1-x)^p\,
\prescript{}{2}{F}_{1}
\left(\frac p2,\frac p2;1;x\right)
=
(1-x)\,
\prescript{}{2}{F}_{1}(-\beta,-\beta;1;x).
\end{equation}
where \({}_2F_1\) denotes the Gauss hypergeometric function.
 
\begin{proposition}\label{prop:unique-radius}
For every \(p>4\), the function \(\mathcal A_p\) has a unique critical point
\(r_p\in(0,1)\). This point is the unique strict global maximum and is
nondegenerate:
\[
 \mathcal A_p''(r_p)<0.
\]
Moreover,
\begin{equation}\label{eq:lower-rp}
 r_p^2>\frac{p-4}{p}.
\end{equation}
\end{proposition}

\begin{proof}
Let
\[
 G(x)=\prescript{}{2}{F}_{1}(-\beta,-\beta;1;x),
 \qquad
 Y_p(x)=(1-x)G(x).
\]
The hypergeometric differential equation for \(G\) is
\[
 x(1-x)G''
 +
 [1-(1-2\beta)x]G'
 -
 \beta^2G
 =0.
\]
Substituting \(G=Y_p/(1-x)\) gives
\begin{equation}\label{eq:Y-ODE}
 \begin{aligned}
 x(1-x)^2Y_p''
 &+
 (1-x)\bigl[1+(2\beta+1)x\bigr]Y_p'\\
 &+
 (\beta+1)\bigl[(\beta+1)x-(\beta-1)\bigr]Y_p
 =0.
 \end{aligned}
\end{equation}
From \eqref{eq:Y-hypergeometric},
\[
 Y_p(0)=1,
 \qquad
 Y_p'(0)=\beta^2-1>0.
\]
Also \(Y_p(x)\to0\) as \(x\to1^{-}\). Hence \(Y_p'\) has at least one
zero. Let \(x_p\) be its first zero. At a critical point,
\eqref{eq:Y-ODE} becomes
\begin{equation}\label{eq:Y-second-critical}
 x(1-x)^2Y_p''(x)
 =
 (\beta+1)^2
 \left(\frac{\beta-1}{\beta+1}-x\right)Y_p(x).
\end{equation}
Since \(Y_p'>0\) immediately to the left of \(x_p\),
\(Y_p''(x_p)\le0\); therefore
\[
 x_p\ge x_0:=\frac{\beta-1}{\beta+1}=\frac{p-4}{p}.
\]

Equality cannot occur. If \(x_p=x_0\), then
\(Y_p'(x_p)=Y_p''(x_p)=0\). Differentiating
\eqref{eq:Y-ODE} and evaluating at \(x_p\) yields
\[
 x_p(1-x_p)^2Y_p'''(x_p)
 +
 (\beta+1)^2Y_p(x_p)
 =0,
\]
so \(Y_p'''(x_p)<0\). Consequently,
\[
 Y_p'(x_p+h)
 =
 \frac12Y_p'''(x_p)h^2+o(h^2)<0
\]
for all sufficiently small nonzero \(h\), contradicting
\(Y_p'>0\) to the left of the first zero. Thus \(x_p>x_0\), and
\eqref{eq:Y-second-critical} gives \(Y_p''(x_p)<0\).

Suppose \(Y_p'\) has a second zero \(x_2>x_p\), and take the first such
zero. Then \(Y_p'<0\) on \((x_p,x_2)\), so differentiability implies
\(Y_p''(x_2)\ge0\). Since \(x_2>x_p>x_0\),
\eqref{eq:Y-second-critical} gives \(Y_p''(x_2)<0\), a contradiction.
Thus \(x_p\) is the unique critical point. Since
\(Y_p'(0)>0\), \(Y_p''(x_p)<0\), and \(Y_p(x)\to0\) at \(1\), it is the
unique strict global maximum.

Set \(r_p=\sqrt{x_p}\). At the critical point,
\[
 \mathcal A_p''(r_p)
 =
 4r_p^2Y_p''(x_p)<0.
\]
The estimate \eqref{eq:lower-rp} follows from \(x_p>x_0\).
\end{proof}
 
\subsection{Positivity of the splitting coefficients}

To split the circle of maximizing centers obtained above, we perturb
the identity function by higher-order monomials. The first variation
of the corresponding \(A_p\)-quantity is governed by the
coefficients defined below. For the perturbation used later, the
crucial point is the positivity of the coefficients corresponding to
the cubic and quartic terms at the maximizing radius \(r_p\).

For an integer \(m\ge2\), set
\begin{equation}\label{eq:Cm}
C_{m,p}(r)
=
\int_{\T}
\abs{\sigma_r(\eta)-r}^{p-2}
\overline{\sigma_r(\eta)-r}
\bigl(\sigma_r(\eta)^m-r^m\bigr)\dm(\eta).
\end{equation}
Reflection symmetry about the real axis gives $
C_{m,p}(r)\in\mathbb R$.

The required signs are established in the following proposition.

\begin{proposition} 
\label{prop:C-positive}
For every \(p>4\),
\[
 C_{3,p}(r_p)>0
 \quad\text{and}\quad
 C_{4,p}(r_p)>0.
\]
\end{proposition}

\begin{proof}
We divide the proof into four steps.

\smallskip
\noindent
\textit{Step 1: }
For \(0<r<1\), consider the Laurent expansion of
\[
|u-r|^{2\beta} =
(1-ru)^\beta(1-ru^{-1})^\beta,
\qquad u\in\mathbb T.
\]
For \(k\geq0\), let \(M_k(r)\) denote the coefficient of
\(u^{-k}\) in this Laurent expansion. Equivalently,
\begin{equation}\label{eq:Mk}
M_k(r)
=
\int_{\mathbb T}|u-r|^{2\beta}u^k\dm(u)
=
\int_{\mathbb T}|u-r|^{p-2}u^k\dm(u).
\end{equation} Since
\(|u-r|^{2\beta}\) is invariant under \(u\mapsto\bar u\), these
coefficients are real and satisfy \(M_{-k}(r)=M_k(r)\).

For \(m\ge2\),
\begin{equation}\label{eq:C-moment-formula}
 C_{m,p}(r)
 =
 (1-r^2)
 \sum_{j=0}^{m-1}r^jM_{m-1-j}(r).
\end{equation}
In particular,
\begin{align}
 C_{2,p}(r)
 &=
 (1-r^2)(M_1+rM_0),\label{eq:C2}\\
 C_{3,p}(r)
 &=
 (1-r^2)(M_2+rM_1+r^2M_0),\label{eq:C3}\\
 C_{4,p}(r)
 &=
 (1-r^2)(M_3+rM_2+r^2M_1+r^3M_0).
 \label{eq:C4}
\end{align}
Factoring \(\sigma_r^m-r^m\) and applying the boundary change of
variables \(u=\sigma_r(\eta)\) give \eqref{eq:C-moment-formula}.

\smallskip
\noindent
\textit{Step 2: }
For \(0<r<1\),
\begin{equation}\label{eq:C2-derivative}
 pC_{2,p}(r)=-(1-r^2)\mathcal A_p'(r).
\end{equation}
Consequently,
\begin{equation}\label{eq:M1-critical}
 M_1(r_p)=-r_pM_0(r_p).
\end{equation}
For real \(s\in(-1,1)\), let
\[
 \tau_s(z)=\frac{z+s}{1+sz}.
\]
For \(s\) sufficiently close to \(0\), one has \(\tau_s(r)>0\), and
the covariance identity gives
\[
    A_p(\tau_s,r)=\mathcal A_p(\tau_s(r)).
\]
Moreover,
\[
 \tau_s(z)=z+s(1-z^2)+O(s^2)
\]
uniformly on compact subsets. The constant term in \(1-z^2\) cancels
in the oscillation difference, so differentiation at \(s=0\) gives
\[
 \left.\frac{d}{ds}A_p(\tau_s,r)\right|_{s=0}
 =
 -pC_{2,p}(r).
\]
On the other hand,
\[
 \left.\frac{d}{ds}\mathcal A_p(\tau_s(r))\right|_{s=0}
 =
 (1-r^2)\mathcal A_p'(r).
\]
This proves \eqref{eq:C2-derivative}. At \(r=r_p\),
\(\mathcal A_p'(r_p)=0\), hence \(C_{2,p}(r_p)=0\). Equation
\eqref{eq:M1-critical} follows from \eqref{eq:C2}.

\smallskip
\noindent
\textit{Step 3:  }  
For \(k\ge1\),
\begin{equation}\label{eq:moment-recurrence}
 r(k+\beta+1)M_{k+1}
 -
 k(1+r^2)M_k
 +
 r(k-\beta-1)M_{k-1}
 =0.
\end{equation}
On \(\T\),
\[
 W_r(u):=\abs{u-r}^{2\beta}
 =
 (1-ru)^\beta(1-r/u)^\beta.
\]
A direct differentiation gives
\[
 (1-ru)(1-r/u)\,uW_r'(u)
 =
 \beta r(u^{-1}-u)W_r(u).
\]
Take the constant Laurent coefficient after multiplication by \(u^k\).
Using that the constant coefficient of
\(u\frac{d}{du}(u^kPW_r)\) vanishes, where
\(P=(1-ru)(1-r/u)\), gives \eqref{eq:moment-recurrence}.

At \(r=r_p\), the case \(k=1\), together with
\eqref{eq:M1-critical}, yields
\begin{equation}\label{eq:M2-critical}
 M_2(r_p)
 =
 \frac{\beta-1-r_p^2}{\beta+2}M_0(r_p).
\end{equation}

For \(k\ge0\),
\begin{equation}\label{eq:Mk-hypergeometric}
 M_k(r)
 =
 \frac{(-\beta)_k}{k!}r^k \cdot 
 \prescript{}{2}{F}_{1}(-\beta,k-\beta;k+1;r^2).
\end{equation}
Equivalently,
\begin{equation}\label{eq:Mk-Euler}
 \begin{aligned}
 M_k(r)
 &=
 \frac{(-\beta)_k}{k!}r^k(1-r^2)^{1+2\beta}\cdot
 \prescript{}{2}{F}_{1}(k+1+\beta,1+\beta;k+1;r^2).
 \end{aligned}
\end{equation}
The hypergeometric factor on the second line has strictly positive
Taylor coefficients. To verify the formula, expand
\[
 (1-ru)^\beta
 =
 \sum_{j\ge0}\frac{(-\beta)_j}{j!}r^ju^j,
 \qquad
 (1-r/u)^\beta
 =
 \sum_{\ell\ge0}\frac{(-\beta)_\ell}{\ell!}r^\ell u^{-\ell}.
\]
The constant coefficient of \(u^kW_r(u)\) is obtained by setting
\(\ell=k+j\), which gives
\[
 M_k(r)
 =
 \frac{(-\beta)_k}{k!}r^k
 \sum_{j=0}^{\infty}
 \frac{(-\beta)_j(k-\beta)_j}{(k+1)_j\,j!}r^{2j}.
\]
This is \eqref{eq:Mk-hypergeometric}. Euler's transformation gives
\eqref{eq:Mk-Euler}.

\smallskip
\noindent
\textit{Step 4: signs of \(C_{3,p}\) and \(C_{4,p}\).}
At \(r=r_p\), equations \eqref{eq:C3} and
\eqref{eq:M1-critical} give
\begin{equation}\label{eq:C3-M2}
 C_{3,p}(r_p)=(1-r_p^2)M_2(r_p).
\end{equation}
Since
\[
 (-\beta)_2=\beta(\beta-1)>0
\]
and the hypergeometric factor in \eqref{eq:Mk-Euler} is positive,
\[
 M_2(r)>0
 \qquad(0<r<1).
\]
Thus \(C_{3,p}(r_p)>0\). Combining this fact with
\eqref{eq:M2-critical} also gives
\begin{equation}\label{eq:upper-rp}
 r_p^2<\beta-1=\frac{p-4}{2}.
\end{equation}

For \(C_{4,p}\), equations \eqref{eq:C4} and
\eqref{eq:M1-critical} give
\begin{equation}\label{eq:C4-reduced}
 C_{4,p}(r_p)
 =
 (1-r_p^2)\bigl(M_3(r_p)+r_pM_2(r_p)\bigr).
\end{equation}
It remains to prove that the quantity in parentheses is positive.

\smallskip
\noindent
\emph{Case 1: \(1<\beta\le2\).}
By \eqref{eq:Mk-Euler},
\[
 \sgn M_3(r)=\sgn(-\beta)_3.
\]
For \(1<\beta\le2\),
\[
 (-\beta)_3=(-\beta)(1-\beta)(2-\beta)\ge0.
\]
Hence \(M_3\ge0\), while \(M_2>0\), and therefore
\(M_3+rM_2>0\).

\smallskip
\noindent
\emph{Case 2: \(2<\beta<3\).}
Set \(x=r^2\). From \eqref{eq:Mk-Euler},
\[
 M_2
 =
 \frac{\beta(\beta-1)}2
 r^2(1-x)^{1+2\beta}D_\beta(x),
\]
where
\[
 D_\beta(x)=\prescript{}{2}{F}_{1}(\beta+3,\beta+1;3;x),
\]
and
\[
 -M_3
 =
 \frac{\beta(\beta-1)(\beta-2)}6
 r^3(1-x)^{1+2\beta}E_\beta(x),
\]
where
\[
 E_\beta(x)=\prescript{}{2}{F}_{1}(\beta+4,\beta+1;4;x).
\]
For every \(n\ge0\), the ratio of the \(n\)-th Taylor coefficient of \(E_\beta\) to the corresponding coefficient of \(D_\beta\) is
\[
\frac{3(n+\beta+3)}{(\beta+3)(n+3)}
\leq1.
\]
Thus \(E_\beta(x)\le D_\beta(x)\) for \(0\le x<1\), and
\[
 \frac{-M_3}{rM_2}
 \le
 \frac{\beta-2}{3}
 <1.
\]
Hence \(M_3+rM_2>0\).

\smallskip
\noindent
\emph{Case 3: \(\beta\ge3\).}
The recurrence \eqref{eq:moment-recurrence} with \(k=2\), together with
\eqref{eq:M1-critical} and \eqref{eq:M2-critical}, gives
\begin{equation}\label{eq:M3-plus}
 M_3(r_p)+r_pM_2(r_p)
 =
 -\frac{M_0(r_p)}
 {r_p(\beta+2)(\beta+3)}
 Q_\beta(r_p^2).
\end{equation}
where
\[
 Q_\beta(x)
 =
 (\beta+5)x^2+(5-3\beta)x+2(1-\beta).
\]
Now
\[
 Q_\beta(0)=2(1-\beta)<0,
 \qquad
 Q_\beta(1)=12-4\beta\le0,
\]
and \(Q_\beta''(x)=2(\beta+5)>0\). Since a convex function lies below
the chord joining its endpoint values,
\[
 Q_\beta(x)<0
 \qquad(0<x<1).
\]
Equation \eqref{eq:M3-plus} therefore yields \(M_3+rM_2>0\).

The three cases prove \(C_{4,p}(r_p)>0\).
\end{proof}

\section{Proof of Theorem~\ref{thm:peak-family}}\label{sec:peak-family-proof}

\begin{proof}[Proof of Theorem~\ref{thm:peak-family}]
	First suppose that \(0<p\le4\). Fix \(w\in\D\) and set
	\(\mathcal P_{p,w}:=\sigma_w\). By the covariance identity
	\eqref{eq:covariance},
	\[
	A_p(\sigma_w,a)=A_p(\sigma_0,\sigma_a(w)).
	\]
	Since \(m\) is normalized, monotonicity of the \(L^q\)-means and
	\cite[Lemma~2.5]{ChenWulan} give
	\[
	\begin{aligned}
		A_p(\sigma_w,a)^{1/p}
		&\le A_4(\sigma_w,a)^{1/4} =\bigl(1-\abs{\sigma_a(w)}^4\bigr)^{1/4}.
	\end{aligned}
	\]
	At \(a=w\), the value equals \(1\). If \(a\ne w\), then
	\(\sigma_a(w)\ne0\), and hence \(A_p(\sigma_w,a)<1\). Therefore
	\[
	B_p(\sigma_w)=1,
	\qquad
	A_p(\sigma_w,w)=1,
	\]
	and \(w\) is the unique maximizing center.
	
	It remains to consider \(p>4\). Fix such a \(p\). We prove that every
	sufficiently small positive \(t\) produces the required peak center. Put
	\[
	F_t(z)=z+t(z^3+z^4),
	\qquad
	\mathcal J(t,a):=A_p(F_t,a).
	\]
	For \(a=re^{i\theta}\), set
	\begin{equation}\label{eq:J-polar}
		\widetilde{\mathcal J}(t,r,\theta)
		:=
		\mathcal J(t,re^{i\theta})
		=
		A_p(F_t,re^{i\theta}).
	\end{equation}
	
	\smallskip
	\noindent
	\textit{Claim 1: first-order angular splitting.}
	For \(a=re^{i\theta}\), differentiation under the integral gives
	\begin{equation}\label{eq:first-variation}
		\left.
		\partial_t\widetilde{\mathcal J}(t,r,\theta)
		\right|_{t=0}
		=
		p\left[
		C_{3,p}(r)\cos2\theta+C_{4,p}(r)\cos3\theta
		\right].
	\end{equation}
	Indeed, after writing \(\zeta=e^{i\theta}\eta\), one has $
	\sigma_{re^{i\theta}}(e^{i\theta}\eta)
	=e^{i\theta}\sigma_r(\eta)$. 
	Thus the monomial \(z^m\) contributes the angular factor
	\(e^{i(m-1)\theta}\), while reflection symmetry makes
	\(C_{m,p}(r)\) real. This proves \eqref{eq:first-variation}.
	
	On the unperturbed maximizing circle, set
	\begin{equation}\label{eq:angular-g}
		g_p(\theta)
		=
		p\bigl(
		C_{3,p}(r_p)\cos2\theta+C_{4,p}(r_p)\cos3\theta
		\bigr).
	\end{equation}
	By Proposition~\ref{prop:C-positive}, both coefficients are positive. Hence
	\[
	g_p(\theta)
	\le
	p\bigl(C_{3,p}(r_p)+C_{4,p}(r_p)\bigr),
	\]
	and equality requires \(\cos2\theta=\cos3\theta=1\). Since
	\(\gcd(2,3)=1\), this occurs only for
	\(\theta\in2\pi\mathbb Z\). Moreover,
	\[
	g_p''(0)
	=
	-p\bigl(4C_{3,p}(r_p)+9C_{4,p}(r_p)\bigr)<0.
	\]
	Thus \(g_p\) has a unique nondegenerate maximum at \(0\).
	
	\smallskip
	\noindent
\textit{Claim 2: uniform \(C^3\) expansion near the maximizing circle.}
For any \(0<\eta<\min\{r_p,1-r_p\}\), define the annular neighborhood
\[
\mathcal N_\eta
:=
\{re^{i\theta}:r_p-\eta\le r\le r_p+\eta\}.
\]
	There exist \(t_0>0\) and \(C>0\) such that, for \(\abs t<t_0\),
	\begin{equation}\label{eq:claim-C3-expansion}
		\widetilde{\mathcal J}(t,r,\theta)
		=
		\mathcal A_p(r)+tG_p(r,\theta)+t^2R_t(r,\theta),
	\end{equation}
	where
	\[
	G_p(r,\theta)
	=
	p\bigl(C_{3,p}(r)\cos2\theta+C_{4,p}(r)\cos3\theta\bigr)
	\]
	and
\[
\sup_{\abs t<t_0}
\norm{R_t}_{C^3([r_p-\eta,r_p+\eta]\times\T)}
\le C.
\]
	If
\[
\mathcal A_p''(r)\le-\kappa<0
\qquad
(r_p-\eta\le r\le r_p+\eta),
\]
then, for sufficiently small \(t\), the equation $
\partial_r\widetilde{\mathcal J}(t,r,\theta)=0$ 
has at most one solution
\(r=r(t,\theta)\in(r_p-\eta,r_p+\eta)\).
If such a solution exists for every \(\theta\in\mathbb T\), then
\begin{equation}\label{eq:claim-r-t}
\|r(t,\cdot)-r_p\|_{C^2(\mathbb T)}
=
O(|t|).
\end{equation}
	
To prove the claim, write $
X(r,\theta,\zeta)
=
\sigma_{re^{i\theta}}(\zeta)-re^{i\theta}$, 
and let \(Y\) denote the corresponding oscillation difference
generated by \(z^3+z^4\). On \(\mathcal N_\eta\times\T\),
\[
|X|
=
\frac{1-r^2}{|1-re^{-i\theta}\zeta|}
\ge 1-r_p-\eta>0.
\]
Hence, for \(|t|\) sufficiently small, \(|X+tY|\) is uniformly
bounded away from zero, while \(X\), \(Y\), and their derivatives up
to order three in \((r,\theta)\) are uniformly bounded. Taylor's
formula and differentiation under the integral therefore yield
\[
\widetilde{\mathcal J}(t,r,\theta)
=
\mathcal A_p(r)+tG_p(r,\theta)+O_{C^3}(t^2)
\]
uniformly on the annulus.

If \(\mathcal A_p''\le-\kappa<0\) there, then for sufficiently small
\(t\),
\[
\partial_{rr}\widetilde{\mathcal J}\le-\kappa/2.
\]
Thus the radial critical point, whenever it exists, is unique. The
implicit function theorem and the preceding \(C^3\)-expansion then
give
\[
\|r(t,\cdot)-r_p\|_{C^2(\T)}=O(|t|),
\]
which proves the claim.
	
	We now complete the global argument. Let $
	M_p=\mathcal A_p(r_p)$. 
	Since \(\mathcal A_p''(r_p)<0\), choose  $
	0<\eta<\frac12\min\{r_p,1-r_p\}$ 
	and \(\kappa>0\) such that
\[
\mathcal A_p''(r)\le-\kappa
\qquad
(r_p-\eta\le r\le r_p+\eta).
\]
    Since \(r_p\) is the unique strict global maximum of
\(\mathcal A_p\), there exists \(\gamma>0\) such that
	\begin{equation}\label{eq:global-gap}
		\mathcal A_p(r)\le M_p-4\gamma
		\quad\text{whenever}\quad
		\abs{r-r_p}\ge\eta.
	\end{equation}
	
	Let \(h(z)=z^3+z^4\). On \(\T\),
	\[
	\abs{\sigma_a(\zeta)-a}\le2,
	\qquad
	\abs{h(\sigma_a(\zeta))-h(a)}\le4.
	\]
	The mean-value inequality for \(s\mapsto\abs{s}^p\) therefore gives
	a constant \(K_p\), independent of \(a\), such that
	\begin{equation}\label{eq:uniform-t}
		\abs{\mathcal J(t,a)-\mathcal J(0,a)}
		\le K_p\abs t
	\end{equation}
	for all \(a\in\D\) and all sufficiently small \(t\).
	Choose \(t>0\) sufficiently small that \(K_pt<\gamma\). From
	the boundary estimate
	\[
	\mathcal J(t,a)
	\le
	\left(\sup_{z\in\overline\D}\abs{F_t'(z)}\right)^p
	\mathcal A_p(\abs a)
	\longrightarrow0
	\qquad(\abs a\to1),
	\]
	the function \(a\mapsto\mathcal J(t,a)\) attains its maximum in
	\(\D\). From \eqref{eq:global-gap}--\eqref{eq:uniform-t}, every global
	maximizer of \(a\mapsto\mathcal J(t,a)\) lies in the annulus
	\begin{equation}\label{eq:annulus}
		\abs{\abs a-r_p}<\eta.
	\end{equation}
	
	  On this closed annulus,
	\[
	\abs{\sigma_a(\zeta)-a}\ge1-\abs a
	\ge1-r_p-\eta>0.
	\]
	By Claim~2, we obtain
	\begin{equation}\label{eq:C2-expansion}
		\widetilde{\mathcal J}(t,r,\theta)
		=
		\mathcal A_p(r)+tG_p(r,\theta)+O_{C^2}(t^2),
	\end{equation}
	uniformly on the annulus, where
	\[
	G_p(r,\theta)
	=
	p\bigl(
	C_{3,p}(r)\cos2\theta+
	C_{4,p}(r)\cos3\theta
	\bigr).
	\]
	For sufficiently small \(t>0\),
\[
\partial_{rr}\widetilde{\mathcal J}(t,r,\theta)
\le-\frac{\kappa}{2}
\]
throughout the annulus. Since \(r_p\) is a strict radial maximum and
\(\mathcal A_p''<0\) on the radial interval,
\[
\mathcal A_p'(r_p-\eta)>0,
\qquad
\mathcal A_p'(r_p+\eta)<0.
\]
By the uniform \(C^1\)-convergence as \(t\to0\), these signs persist:
\[
\partial_r\widetilde{\mathcal J}
(t,r_p-\eta,\theta)>0,
\qquad
\partial_r\widetilde{\mathcal J}
(t,r_p+\eta,\theta)<0
\]
uniformly in \(\theta\). Hence the intermediate value theorem gives,
for every \(\theta\in\mathbb T\), a solution
\[
r(t,\theta)\in(r_p-\eta,r_p+\eta)
\]
of
\[
\partial_r\widetilde{\mathcal J}
\bigl(t,r(t,\theta),\theta\bigr)=0.
\]
The strict radial concavity gives uniqueness. Claim~2 therefore yields
\begin{equation}\label{eq:r-t}
\|r(t,\cdot)-r_p\|_{C^2(\mathbb T)}
=
O(t).
\end{equation}

	Define the reduced angular function $
	K(t,\theta)
	=
	\widetilde{\mathcal J}\bigl(t,r(t,\theta),\theta\bigr)$. 
	Using \eqref{eq:C2-expansion}, \eqref{eq:r-t}, and
	\(\mathcal A_p'(r_p)=0\), we obtain
	\begin{equation}\label{eq:reduced-expansion}
		K(t,\theta)
		=
		M_p+t g_p(\theta)+O_{C^2}(t^2).
	\end{equation}
	By Claim~1, choose a neighborhood \(U\) of \(0\) and
	constants \(c_p,d_p>0\) such that
	\[
	g_p''(\theta)\le-c_p
	\quad(\theta\in U),
	\]
	and
	\[
	g_p(0)-g_p(\theta)\ge d_p
	\quad(\theta\notin U).
	\]
	Equation \eqref{eq:reduced-expansion} implies, for sufficiently small
	positive \(t\),
	\[
	K(t,0)>K(t,\theta)
	\quad(\theta\notin U)
	\]
	and
	\[
	\partial_{\theta\theta}K(t,\theta)<0
	\quad(\theta\in U).
	\]
	Because \(F_t\) has real coefficients,
	\[
	\mathcal J(t,\overline a)=\mathcal J(t,a).
	\]
	The unique radial maximizer therefore satisfies $
	r(t,-\theta)=r(t,\theta)$, 
	so \(K(t,\cdot)\) is even and $
	\partial_\theta K(t,0)=0$. 
	Strict concavity on \(U\) now shows that \(0\) is its unique maximum
	there. It is also the unique global maximum by the strict gap outside
	\(U\). Hence $
	a_t=r(t,0)\in(0,1)$ 
	is the unique global maximizer of \(a\mapsto\mathcal J(t,a)\).
	
	Finally, at \((t,r(t,0),0)\) one has $
	\partial_{rr}\widetilde{\mathcal J}(t,r(t,0),0)<0$. 
	Since $
	\partial_r\widetilde{\mathcal J}
	\bigl(t,r(t,\theta),\theta\bigr)=0$, 
	implicit differentiation gives the  identity
	\[
	\partial_{\theta\theta}K(t,0)
	=
	\partial_{\theta\theta}\widetilde{\mathcal J}(t,r(t,0),0)
	-
	\frac{
		\bigl(
		\partial_{r\theta}\widetilde{\mathcal J}(t,r(t,0),0)
		\bigr)^2
	}{
		\partial_{rr}\widetilde{\mathcal J}(t,r(t,0),0)
	}<0,
	\]
	while $
	\partial_{rr}\widetilde{\mathcal J}(t,r(t,0),0)<0$. 
	The Hessian is therefore negative definite in the regular polar
	coordinates \(r>0\), and the maximum is nondegenerate.
	
	Set
	\[
	\tau_p=t,
	\qquad
	\rho_p=a_t,
	\qquad
	F_p=F_t,
	\qquad
	\Lambda_p=A_p(F_p,\rho_p)^{1/p}=B_p(F_p).
	\]
	For \(w\in\D\), define \(\mathcal P_{p,w}\) by
	\eqref{eq:high-p-peak}. The covariance identity
	\eqref{eq:covariance} gives
	\[
	A_p(\mathcal P_{p,w},a)
	=
	\Lambda_p^{-p}
	A_p\bigl(F_p,\sigma_{\rho_p}(\sigma_w(a))\bigr).
	\]
	This is maximal precisely when
	\[
	\sigma_{\rho_p}(\sigma_w(a))=\rho_p,
	\]
	equivalently when \(a=w\). The factor \(\Lambda_p^{-1}\) makes the
	peak value and the quasi-norm equal to \(1\), while
	nondegeneracy is preserved by the automorphisms and the nonzero scalar
	factor. This completes the proof of Theorem~\ref{thm:peak-family}.
\end{proof}

\section{Proof of  Theorem \ref{thm:main}}\label{sec:main-proof}

\begin{proof}[Proof of Theorem \ref{thm:main}] 
	The implication
	\textup{(i)}\(\Rightarrow\)\textup{(ii)}
	follows from Proposition~\ref{prop:seminorm-reduction}. Indeed, since
	\(\varphi(0)=0\), condition \textup{(i)} gives
	\[
	B_p(f\circ\varphi)=B_p(f)
	\qquad(f\in\BMOA).
	\]
	Taking \(f=\mathcal P_{p,w}\) and using
	\(B_p(\mathcal P_{p,w})=1\), we obtain condition
	\textup{(ii)}.

We next prove \textup{(ii)}\(\Rightarrow\)\textup{(iv)}. Assume
\textup{(ii)}.

Fix \(w\in\D\), put \(P=\mathcal P_{p,w}\), and choose
\(\{a_n\}_{n\ge 1}\subset\D\) so that
\[
 A_p(P\circ\varphi,a_n)
 \longrightarrow
 B_p^p(P\circ\varphi)
 =
 B_p^p(P)
 =1.
\]
Write \(b_n=\varphi(a_n)\). By \eqref{eq:defect-formula},
\begin{equation}\label{eq:peak-test-defect}
 A_p(P\circ\varphi,a_n)
 =
 A_p(P,b_n)-\Delta_{p,P,\varphi}(a_n).
\end{equation}
Since
\[
 0\le A_p(P,b_n)\le1,
 \qquad
 \Delta_{p,P,\varphi}(a_n)\ge0,
\]
we obtain
\begin{equation}\label{eq:peak-test-two-limits}
 A_p(P,b_n)\longrightarrow1,
 \qquad
 \Delta_{p,P,\varphi}(a_n)\longrightarrow0.
\end{equation}
Since \(P\) is analytic on a neighborhood of \(\overline\D\), there
exists \(L>0\) such that
\[
A_p(P,a)
\le
L^p\mathcal A_p(|a|)
\longrightarrow0
\qquad (|a|\to1).
\]
Hence \(a\mapsto A_p(P,a)\) extends continuously to
\(\overline\D\) with boundary value \(0\). Since its unique maximum is
attained at \(w\), the convergence \(A_p(P,b_n)\to1\) forces
\(b_n\to w\).

Since \(A_p(P,w)=1\), the function \(P\) is nonconstant. Hence $
Z_w=\{P'=0\}\cup\{P=P(w)\}$ 
is discrete.  If \(K\Subset\D\setminus Z_w\), then
\(b_n\to w\) implies that
\[
 \abs{P(z)-P(b_n)}^{p-2}\abs{P'(z)}^2
 \longrightarrow
 \abs{P(z)-P(w)}^{p-2}\abs{P'(z)}^2
\]
uniformly on \(K\). The limiting function is continuous and strictly
positive there. This observation also covers \(0<p<2\), because
\(\abs{P(z)-P(b_n)}\) is bounded above and bounded away from zero on
\(K\) for all sufficiently large \(n\). Hence there exists \(c_K>0\)
such that
\[
 \abs{P(z)-P(b_n)}^{p-2}\abs{P'(z)}^2\ge c_K
 \qquad(z\in K)
\]
for all sufficiently large \(n\). By \eqref{eq:defect-integral},
\[
 \Delta_{p,P,\varphi}(a_n)
 \ge
 \frac{p^2c_K}{2\pi}
 \int_KH_\varphi(z,a_n)\dA(z).
\]
Together with \eqref{eq:peak-test-two-limits}, this gives
\[
 \int_KH_\varphi(z,a_n)\dA(z)\longrightarrow0,
\]
for every compact \(K\Subset\D\setminus Z_w\).
By a standard diagonal argument, after passing to a subsequence we
obtain
\[
H_\varphi(z,a_n)\longrightarrow0
\quad\text{for almost every }z\in\D\setminus Z_w.
\]
Since \(Z_w\) is discrete, it has area zero, and condition
\textup{(iv)} follows.
This proves \textup{(ii)}\(\Rightarrow\)\textup{(iv)}.

The equivalence
\textup{(iii)}\(\Leftrightarrow\)\textup{(iv)}
follows from \cite[Theorem~1.1]{ChenWulan}, applied with exponent
\(2\), together with Lemma~\ref{lem:exponent-transfer}.

It remains to prove \textup{(iv)}\(\Rightarrow\)\textup{(i)}, using
the argument recorded in \cite[Remark~2.1]{ChenWulan}.
Fix \(f\in\BMOA\) and \(w\in\D\). If \(f\) is constant, there is nothing to prove. We therefore assume
that \(f\) is nonconstant. Choose
\(\{a_n\}_{n\ge 1}\) as in \textup{(iv)} and set \(b_n:=\varphi(a_n)\). Then
\(b_n\to w\), and
\[
 0\le H_\varphi(z,a_n)\le G(z,b_n).
\]

We claim that \(\Delta_{p,f,\varphi}(a_n)\to0\). Set
\[
 U_n(z)
 :=
 \abs{f(z)-f(b_n)}^{p-2}\abs{f'(z)}^2G(z,b_n),
\]
and let \(U\) be the analogous function with \(b_n\) replaced by
\(w\). Then \(U_n\to U\) almost everywhere. By the Hardy--Stein
identity and continuity of \(b\mapsto A_p(f,b)\),
\[
 \int_{\D}U_n\dA
 =
 \frac{2\pi}{p^2}A_p(f,b_n)
 \longrightarrow
 \frac{2\pi}{p^2}A_p(f,w)
 =
 \int_{\D}U\dA.
\]
  Since \(U_n,U\geq0\), \(U_n\to U\) almost everywhere, and
\[
\int_{\D}U_n\,\dA\longrightarrow\int_{\D}U\,\dA,
\]
it follows that \(U_n\to U\) in \(L^1(\D)\). Consequently, \((U_n)\) is uniformly integrable. Now set
\[
 V_n(z)
 :=
 \abs{f(z)-f(b_n)}^{p-2}\abs{f'(z)}^2
 H_\varphi(z,a_n).
\]
Then \(0\le V_n\le U_n\) and \(V_n\to0\) almost everywhere. Vitali's
theorem yields \(\int_{\D}V_n\dA\to0\), which is precisely
\(\Delta_{p,f,\varphi}(a_n)\to0\).

The defect identity now gives
\[
 A_p(f\circ\varphi,a_n)
 =
 A_p(f,b_n)-\Delta_{p,f,\varphi}(a_n)
 \longrightarrow
 A_p(f,w).
\]
Hence \(B_p^p(f\circ\varphi)\ge A_p(f,w)\). Taking the supremum over
\(w\in\D\) gives \(B_p(f\circ\varphi)\ge B_p(f)\). Together with
\eqref{eq:subordination}, this proves equality. Since
\(\varphi(0)=0\), Proposition~\ref{prop:seminorm-reduction} shows that the full
quasi-norm is preserved. 
The proof is complete.
\end{proof}

  \section*{Acknowledgments}
The author is grateful to Professor Hasi Wulan for an inspiring talk on
composition operators on BMOA, which drew the author's attention to the
problem considered in this paper.

\bibliographystyle{amsplain}
\bibliography{references}

\medskip
 \noindent
School of Mathematical Sciences, Dalian University of Technology,
Dalian, Liaoning 116024, P. R. China

\noindent
Email address: \texttt{linzhaopeng2606@163.com} (Zhaopeng Lin)

\end{document}